\documentclass[leqno,11pt]{article}

\newif\ifpaper
\newif\ifarxiv

\papertrue
\arxivtrue

\usepackage[utf8]{inputenc}
\usepackage[T1]{fontenc}
\usepackage{microtype}

\ifpaper
  \usepackage[
    letterpaper
  ]{geometry}
\else
  \usepackage[
    a5paper,
    margin=2em,
    includefoot
  ]{geometry}
\fi

\usepackage{amsmath}
\usepackage{amsthm}
\usepackage{amssymb}
\usepackage{mathrsfs}

\usepackage[english]{babel}
\usepackage{csquotes}

\ifarxiv
  \usepackage{ae,aecompl}
  \usepackage{times}
\else
  \usepackage{newtxtext,newtxmath}
\fi

\ifarxiv
  \usepackage[colorlinks,pagebackref,hypertexnames=false]{hyperref}
  \usepackage[numbers]{natbib}
  
\else
  \usepackage[
    backend=biber,
    hyperref=true,
    backref=true,
    isbn=false,
    doi=false,
    natbib=true,
    eprint=false,
    useprefix=true,
    maxcitenames=99, % effectively disable et. al
    maxbibnames=99,
    safeinputenc
  ]{biblatex}
  \usepackage[colorlinks,hypertexnames=false]{hyperref}
\fi

\usepackage{esint} 

\usepackage[yyyymmdd]{datetime}

\newcommand{\printreferences}{
  \ifarxiv        
    \bibliographystyle{plainnat}
    \bibliography{pre/refs}
   \else
    \printbibliography[heading=bibintoc]
  \fi
}

\newcommand{\N}{\bN}

\newcommand{\R}{\bR}
\newcommand{\C}{\bC}

\newcommand{\su}{\mathfrak{su}}

\newcommand{\gl}{\mathfrak{gl}}

\newcommand{\U}{{\rm U}}
\newcommand{\PU}{{\bP\U}}

\newcommand{\Vect}{\mathrm{Vect}}

\DeclareMathOperator{\Hom}{Hom}

\DeclareMathOperator{\ad}{ad}

\DeclareMathOperator{\rk}{rk}

\DeclareMathOperator{\tr}{tr}

\newcommand{\End}{\mathop{\rm End}\nolimits}

\renewcommand{\det}{\operatorname{det}}

\newcommand{\delbar}{\bar{\del}}
\newcommand{\del}{\partial}

\newcommand{\id}{{\rm id}}
\newcommand{\loc}{{\rm loc}}

\newcommand{\sing}{{\rm sing}}
\newcommand{\vol}{{\rm vol}}
\renewcommand{\P}{\bP}
\renewcommand{\epsilon}{\varepsilon}

\def\({\left(}
\def\){\right)}
\def\<{\left\langle}
\def\>{\right\rangle}

\newcommand{\co}{\mskip0.5mu\colon\thinspace}

\newcommand{\iso}{\cong}

\newcommand{\andq}{\text{and}\quad}
\newcommand{\qandq}{\quad\text{and}\quad}

\newcommand{\qforeach}{\quad\text{for each }}

\usepackage{mathtools}
\DeclarePairedDelimiter{\Abs}{\|}{\|}
\DeclarePairedDelimiter{\abs}{\lvert}{\rvert}

\DeclarePairedDelimiter{\set}{\lbrace}{\rbrace}
\DeclarePairedDelimiter{\braket}{\langle}{\rangle}

\newcommand{\inner}[2]{\braket{#1, #2}}

\def\MSC#1{\href{http://www.ams.org/msc/msc2010.html?&s=#1}{#1}}

\makeatletter
\renewcommand\xleftrightarrow[2][]{%
  \ext@arrow 9999{\longleftrightarrowfill@}{#1}{#2}}
\newcommand\longleftrightarrowfill@{%
  \arrowfill@\leftarrow\relbar\rightarrow}
\makeatother

\newcommand{\rd}{{\rm d}}

\newcommand{\rF}{{\rm F}}

\newcommand{\rK}{{\rm K}}

\newcommand{\bg}{\mathbf{g}}

\newcommand{\bC}{\mathbf{C}}

\newcommand{\bN}{\mathbf{N}}

\newcommand{\bP}{\mathbf{P}}

\newcommand{\bR}{\mathbf{R}}

\newcommand{\sE}{\mathscr{E}}
\newcommand{\sF}{\mathscr{F}}
\newcommand{\sG}{\mathscr{G}}

\newcommand{\sL}{\mathscr{L}}

\newcommand{\sO}{\mathscr{O}}

\newcommand{\sT}{\mathscr{T}}

\newcommand{\fg}{{\mathfrak g}}

\newcommand{\fu}{{\mathfrak u}}

\usepackage{slashed}

\numberwithin{equation}{section}

\renewcommand{\eqref}[1]{\hyperref[#1]{\rm(\ref*{#1})}}

\usepackage{aliascnt}

\def\makeautorefname#1#2{\AtBeginDocument{\expandafter\def\csname#1autorefname\endcsname{#2}}}

\newcommand{\mynewtheorem}[2]{
  \newaliascnt{#1}{equation}          
  \newtheorem{#1}[#1]{#2}
  \aliascntresetthe{#1}
  \makeautorefname{#1}{#2}
}

\newcommand{\mynewproblem}[2]{
  \newaliascnt{#1}{myProblem}          
  \newtheorem{#1}[#1]{#2}
  \aliascntresetthe{#1}
  \makeautorefname{#1}{#2}
}

\mynewtheorem{theorem}{Theorem}
\newtheorem*{theorem*}{Theorem}
\mynewtheorem{prop}{Proposition}
\newtheorem*{prop*}{Proposition}

\mynewtheorem{cor}{Corollary}
\mynewtheorem{construction}{Construction}
\mynewtheorem{lemma}{Lemma}
\mynewtheorem{conjecture}{Conjecture}
\newtheorem*{conjecture*}{Conjecture}
\mynewtheorem{hyp}{Hypothesis}
\newtheorem{step}{Step}

\numberwithin{substep}{step}
\makeautorefname{step}{Step}
\makeautorefname{substep}{Step}

\numberwithin{subcase}{case}
\makeautorefname{case}{Case}
\makeautorefname{subcase}{case}

\theoremstyle{remark}
\mynewtheorem{remark}{Remark}
\newtheorem*{remark*}{Remark}

\theoremstyle{definition}
\mynewtheorem{definition}{Definition}
\mynewtheorem{example}{Example}
\mynewtheorem{exercise}{Exercise}
\mynewproblem{problem}{Problem}
\mynewtheorem{solution}{Solution}
\mynewtheorem{convention}{Convention}
\newtheorem*{convention*}{Convention}
\newtheorem*{conventions*}{Conventions}
\mynewtheorem{question}{Question}
\newtheorem*{question*}{Question}

\makeautorefname{table}{Table}        
\makeautorefname{chapter}{Chapter}
\makeautorefname{section}{Section}
\makeautorefname{subsection}{Section}
\makeautorefname{subsubsection}{Section}
\makeautorefname{footnote}{Footnote}

\ifarxiv
\else
  \usepackage{tikz-cd}
  \usetikzlibrary{arrows} 
  \tikzset{
    commutative diagrams/.cd, 
    arrow style=tikz, 
    diagrams={>=stealth}
  }
\fi

\author{
  Adam Jacob \\
  UC Davis
  \and
  Henrique Sá Earp \\
  Unicamp
  \and
  Thomas Walpuski \\
  Massachusetts Institute of Technology
}
\title{Tangent cones of Hermitian Yang--Mills connections with isolated singularities}
\date{2017-01-10}

\begin{document}

\maketitle

\begin{abstract}
  We give a simple direct proof of uniqueness of tangent cones for singular projectively Hermitian Yang--Mills connections on reflexive sheaves at isolated singularities modelled on a sum of $\mu$--stable holomorphic bundles over $\P^{n-1}$.
\end{abstract}

{\small
  \medskip
  \noindent
  \textbf{MSC2010:}
    \MSC{53C07}; \MSC{14J60}, \MSC{32Q15}
}

\section{Introduction}
\label{Sec_Introduction}

A projectively Hermitian Yang--Mills ($\P$HYM) connection $A$ over a Kähler manifold $X$ is a unitary connection $A$ on a Hermitian vector bundle $(E,H)$ over $X$ satisfying
\begin{equation}
  \label{Eq_PHYMConnection}
  \rF_A^{0,2} = 0 \qandq
  i\Lambda \rF_A - \frac{\tr(i\Lambda \rF_A)}{\rk E}\cdot \id_E = 0.
\end{equation}
Since $\rF_A^{0,2} = 0$, $\sE := (E,\delbar_A)$ is a holomorphic vector bundle and $A$ is the Chern connection of $H$.
A Hermitian metric $H$ on a holomorphic vector bundle is called $\P$HYM if its Chern connection $A_H$ is $\P$HYM.
The celebrated Donaldson--Uhlenbeck--Yau Theorem \cite{Donaldson1985,Donaldson1986,Uhlenbeck1986} asserts that a holomorphic vector bundle $\sE$ on a compact Kähler manifold admits a $\P$HYM metric if and only if it is $\mu$--polystable;
moreover, any two $\P$HYM metrics are related by an automorphism of $\sE$ and by multiplication with a conformal factor.
If $H$ is a $\P$HYM metric, then the connection $A^\circ$ on $\PU(E,H)$ induced by $A_H$ is Hermitian Yang--Mills (HYM), that is, it satisfies $\rF_{A^\circ}^{0,2} = 0$ and $i\Lambda \rF_{A^\circ} = 0$;
it depends only on the conformal class of $H$.
Conversely, any HYM connection $A^\circ$ on $\PU(E,H)$ can be lifted to a $\P$HYM connection $A$; any two choices of lifts lead to isomorphic holomorphic vector bundles $\sE$ and conformal metrics $H$.

An admissible $\P$HYM connection is a $\P$HYM connection $A$ on a Hermitian vector bundle $(E,H)$ over $X \setminus \sing(A)$ with $\sing(A)$ a closed subset with locally finite $(2n-4)$--dimensional Hausdorff measure and $\rF_A \in L^2_\loc(X)$.
\citet{Bando1991} proved that if $A$ is an admissible $\P$HYM connection, then $(E,\delbar_A)$ extends to $X$ as a reflexive sheaf $\sE$ with $\sing(\sE) \subset \sing(A)$.
\citet{Bando1994} proved that a reflexive sheaf on a compact Kähler manifold admits an admissible $\P$HYM metric if and only if it is $\mu$--polystable.

The technique used by Bando and Siu does not yield any information on the behaviour of the admissible $\P$HYM connection $A_H$ near the singularities of the reflexive sheaf $\sE$---not even at isolated singularities.
The simplest example of a reflexive sheaf on $\C^n$ with an isolated singularity at $0$ is $i_*\sigma^*\sF$ with $\sF$ a holomorphic vector bundle over $\P^{n-1}$;
cf.~\citet[Example 1.9.1]{Hartshorne1980}.
\ifarxiv
  Here 
  \begin{equation*}
    i\co \C^n\setminus\set{0} \to \C^n,
    \pi\co \C^n\setminus\set{0} \to S^{2n-1},
    \rho\co S^{2n-1} \to \P^{n-1} \qandq 
    \sigma\co \C^n\setminus\set{0} \to \P^{n-1}
  \end{equation*}
  are the obvious maps.
\else
  Here we use the obvious maps summarised in the following diagram:
  \begin{equation*}
    \begin{tikzcd}[column sep=scriptsize]
      \C^n &
      \C^n\setminus\set{0} \ar[l,swap,hook',"i"] \ar[r,twoheadrightarrow,"\pi"] \ar[rr,twoheadrightarrow,bend right,"\sigma"] &
      S^{2n-1} \ar[r,twoheadrightarrow,"\rho"] &
      \P^{n-1}.
    \end{tikzcd}
  \end{equation*}
\fi
The main result of this article gives a description of $\P$HYM connections near singularities modelled on $i_*\sigma^*\sF$ with $\sF$ a sum of $\mu$--stable holomorphic vector bundles.

\begin{theorem}
  \label{Thm_HYMConnectionTangentCones}
  Let $\omega = \frac{1}{2i} \delbar\del\abs{z}^2 + O(\abs{z}^2)$ be a Kähler form on $\bar B_R(0) \subset \C^n$.
  Let $A$ be an admissible $\P$HYM connection on a Hermitian vector bundle $(E,H)$ over $B_R(0)\setminus\set{0}$ with $\sing(A) = \set{0}$ and $(E,\delbar_A) \iso \sigma^*\sF$ for some holomorphic vector bundle $\sF$ over $\P^{n-1}$.
  Denote by $F$ the complex vector bundle underlying $\sF$.
  If $\sF$ is sum of $\mu$--stable holomorphic vector bundles, then there exists a Hermitian metric $K$ on $F$, a connection $A_*$ on $\sigma^*(F,K)$ which is the pullback of a connection on $\rho^*(F,K)$, and an isomorphism $(E,H) \iso \sigma^*(F,K)$ such that with respect to this isomorphism we have
  \begin{equation*}
    \abs{z}^{k+1}\abs{\nabla_{A_*}^k(A^\circ - A_*^\circ)}
    \leq
      C_k(-\log\abs{z})^{-1/2}
    \qforeach k \geq 0; 
  \end{equation*}
  moreover, if $\sF$ is simple, then
  \begin{equation*}
    \abs{z}^{k+1}\abs{\nabla_{A_*}^k(A^\circ - A_*^\circ)}
    \leq
      D_k\abs{z}^{\alpha}
    \qforeach k \geq 0.
  \end{equation*}  
  The constants $C_k,D_k,\alpha > 0$ depend on $\omega$, $\sF$, $A|_{B_R(0)\setminus B_{R/2}(0)}$, and $\Abs{\rF_A}_{L^2(B_R(0))}$.
\end{theorem}

\begin{remark}
  Using a gauge theoretic Łojasiewicz--Simon gradient inequality, \citet[Theorem 1]{Yang2002} proved that the tangent cone to a stationary Yang--Mills connection---in particular, a $\PU(r)$ HYM connection---with an isolated singularity at $x$ is unique provided
  \begin{equation*}
    \abs{\rF_A} \lesssim d(x,\cdot)^{-2}.
  \end{equation*}
  In our situation, such a curvature bound can be obtained from \autoref{Thm_HYMConnectionTangentCones};
  our proof of this result, however, proceeds more directly---without making use of Yang's theorem.
\end{remark}

The hypothesis that $\sF$ be a sum of $\mu$--stable holomorphic vector bundles is optimal.
This is a consequence of the following observation, which will be proved in \autoref{Sec_RadiallyInvariantHYMConnections}.

\begin{prop}
  \label{Prop_RadiallyInvariantHYMConnections}
  Let $(F,K)$ be a Hermitian vector bundle over $\P^{n-1}$.
  If $B$ is a unitary connection on $\rho^*(F,K)$ such that $A_* := \pi^*B$ is HYM with respect to $\omega_0 := \frac{1}{2i}\delbar\del \abs{z}^2$,
  then there is a $k \in \N$ and, for each $j \in \set{1,\ldots,k}$, $\mu_j \in \R$, a Hermitian vector bundle $(F_j,K_j)$ on $\P^{n-1}$, and an irreducible unitary connection $B_j$ on $F_j$ satisfying
  \begin{equation*}
    \rF_{B_j}^{0,2} = 0 \qandq
    i\Lambda \rF_{B_j} = (2n-2)\pi \mu_j \cdot \id_{F_j}
  \end{equation*}
  such that
  \begin{equation*}
    F = \bigoplus_{j=1}^k F_j \qandq
    B = \bigoplus_{j=1}^k \rho^*B_j + i \mu_j \,\id_{\rho^*F_j}\cdot \theta.
  \end{equation*}
  Here $\theta$ denotes the standard contact structure%
  \footnote{%
    With respect to standard coordinates on $\C^n$, the standard contact structure $\theta$ on $S^{2n-1}$ is such that $\pi^*\theta = \sum_{j=1}^n (\bar z_j\rd z_j - z_j\rd \bar z_j)/2i\abs{z}^2$.
  }
  on $S^{2n-1}$.
  In particular,
  \begin{equation*}
    \sE = (\sigma^*F,\delbar_{A_*}) \iso \bigoplus_{j=1}^k \sigma^*\sF_j
  \end{equation*}
  with $\sF_j = (F_j,\delbar_{B_j})$ $\mu$--stable.
\end{prop}

To conclude the introduction we discuss two concrete examples in which \autoref{Thm_HYMConnectionTangentCones} can be applied.

\begin{example}[{\citet[Example 1.1.13]{Okonek2011}}]
  \label{Ex_MuStable}
  It follows from the Euler sequence that $H^0(\sT_{\P^3}(-1)) \iso \C^4$.
  Denote by $s_v \in H^0(\sT_{\P^3}(-1))$ the section corresponding to $v \in \C^4$.
  If $v \neq 0$, then the rank two sheaf $\sE = \sE_v$ defined by
  \begin{equation*}
    0
    \to
    \sO_{\P^3}
    \xrightarrow{s_v}
    \sT_{\P^3}(-1)
    \to
    \sE_v
    \to
    0
  \end{equation*}
  is reflexive and $\sing(\sE) = \set{[v]}$.

  $\sE$ is $\mu$--stable.
  To see this, because $\mu(\sE) = 1/2$, it suffices to show that
  \begin{equation*}
    \Hom(\sO_{\P^3}(k),\sE) = H^0(\sE(-k)) = 0 \qforeach k \geq 1.
  \end{equation*}
  However, by inspection of the Euler sequence, $H^0(\sE(-k)) \iso H^0(\sT_{\P^3}(-k-1)) = 0$.
  It follows that $\sE$ admits a $\P$HYM metric $H$ with $\rF_H \in L^2$ and a unique singular point at $[v] \in \P^3$.
  To see that \autoref{Thm_HYMConnectionTangentCones} applies, pick a standard affine neighbourhood $U \iso \C^3$ in which $[v]$ corresponds to $0$.
  In $U$, the Euler sequence becomes
  \begin{equation*}
    0 \to \sO_{\C^3} \xrightarrow{(1,z_1,z_2,z_3)} \sO_{\C^3}^{\oplus 4} \to \sT_{\P^3}(-1)|_{U} \to 0,
  \end{equation*}
  and $s_v = [(1,0,0,0)]$;
  hence,
  \begin{equation*}
    0 \to \sO_{\C^3} \xrightarrow{(z_1,z_2,z_3)} \sO_{\C^3}^{\oplus 3} \to \sE_v|_{U} \to 0.
  \end{equation*}
  On $\C^3\setminus\set{0}$, this is the pullback of the Euler sequence on $\P^2$;
  therefore, $\sE_v|_U \iso i_*\sigma^* \sT_{\P^2}$.
\end{example}

\begin{example}
  \label{Ex_MuUnstable}
  For $t \in \C$ define $f_t \co \sO_{\P^3}(-2)^{\oplus 2} \to \sO_{\P^3}(-1)^{\oplus 5}$ by
  \begin{equation*}
    f_t :=
    \begin{pmatrix}
      z_0 & 0 \\
      z_1 & z_0 \\
      z_2 & z_1 \\
      t\cdot z_3 & z_2 \\
      0 & z_3
    \end{pmatrix}
  \end{equation*}
  and denote by $\sE_t$ the cokernel of $f_t$, i.e.,
  \begin{equation}
    \label{Eq_EtSequence}
    0 \to \sO_{\P^3}(-2)^{\oplus 2} \xrightarrow{f_t} \sO_{\P^3}(-1)^{\oplus 5} \to \sE_t \to 0.
  \end{equation}
  If $t \neq 0$, then $\sE_t$ is locally free;
  $\sE_0$ is reflexive with $\sing(\sE_0) = \set{ [0:0:0:1] }$.
  The proof of this is analogous to that of the reflexivity of $\sE_v$ from \autoref{Ex_MuStable} given in \cite[Example 1.1.13]{Okonek2011}.

  For each $t$, $H^0(\sE_t) = H^0(\sE_t^*(-1)) = 0$;
  hence, $\sE_t$ is $\mu$--stable according to the criterion of \citet[Remark 1.2.6(b)]{Okonek2011}.
  The former vanishing is obvious since $H^0(\sO_{\P^3}(-1)) = H^1(\sO_{\P^3}(-2)) = 0$.
  The latter follows by dualising \eqref{Eq_EtSequence}, twisting by $\sO_{\P^3}(-1)$ and observing that the induced map $H^0(f_0^*)\co H^0(\sO_{\P^3})^{\oplus 5} \to H^0(\sO_{\P^3}(1))^{\oplus 2}$, which is given by
  \begin{equation*}
    \begin{pmatrix}
      z_0 & z_1 & z_2 & t\cdot z_3  & 0 \\
      0 & z_0 & z_1 & z_2 & z_3
    \end{pmatrix},
  \end{equation*}
  is injective.

  In a standard affine neighbourhood $U \iso \C^3$ of $[0:0:0:1]$, we have $\sE_0|_U \iso i_*\sigma^*(\sT_{\P^2}\oplus \sO_{\P^2}(1))$.
  To see this, note that the cokernel of the map $g \co \sO_{\P^2}^{\oplus 2} \to \sO_{\P^2}(1)^{\oplus 4} \oplus \sO_{\P^2}$ defined by
  \begin{equation*}
    g :=
    \begin{pmatrix}
      z_0 & 0 \\
      z_1 & z_0 \\
      z_2 & z_1 \\
      0   & z_2 \\
      0   & 1
    \end{pmatrix}
  \end{equation*}
  is $\sT_{\P^2}\oplus \sO_{\P^2}(1)$.
\end{example}

%%% Local Variables:
%%% mode: latex
%%% TeX-master: "HYMTangentCones"
%%% End:

\section{Reduction to the metric setting}

In the situation of \autoref{Thm_HYMConnectionTangentCones}, the Hermitian metric $H$ on $\sE$ corresponds to a $\P$HYM metric on $\sigma^*\sF$ via the isomorphism $(E,\delbar_A) \iso \sigma^*\sF$.
By slight abuse of notation, we will denote this metric by $H$ as well.

Denote by $\sF_1,\ldots, \sF_k$ the $\mu$--stable summands of $\sF$.
Denote by $K_j$ the $\P$HYM metric on $\sF_j$ with
\begin{equation*}
  i\Lambda_{\omega_{FS}} \rF_{K_j}
   = \lambda_j \cdot \id_{F_j}
  := \frac{2\pi}{(n-2)!\vol(\P^{n-1})} \mu_j \cdot \id_{F_j}
   = (2n-2) \pi \mu_j \cdot \id_{F_j}
\end{equation*}
with $\omega_{FS}$ denoting the integral Fubini study form and for $\mu_j := \mu(\sF_j)$.
The Kähler form $\omega_0$ associated with the standard Kähler metric on $\C^n$ can be written as
\begin{equation}
  \label{Eq_Omega0Expansion}
  \omega_0
  = \frac{1}{2i} \delbar\del \abs{z}^2
  = \pi r^2 \sigma^*\omega_{FS} + r \rd r \wedge \pi^*\theta
\end{equation}
with $\theta$ as in \autoref{Prop_RadiallyInvariantHYMConnections}.
Therefore, we have
\begin{equation*}
  i\Lambda_{\omega_0} \rF_{\sigma^*K_j}
  = (2n-2) \mu_j r^{-2} \cdot \id_{\sigma^*F_j},
\end{equation*}
and $H_{\diamond,j} := r^{2\mu_j} \cdot \sigma^*K_j$ satisfies
\begin{align*}
  i\Lambda_{\omega_0} \rF_{H_{\diamond,j}}
  &=
    i\Lambda_{\omega_0} \rF_{\sigma^* K_j}
    + i\Lambda_{\omega_0} \delbar\del \log r^{2\mu_j} \cdot \id_{\sigma^*F_j} \\
  &=
    i\Lambda_{\omega_0} \rF_{\sigma^* K_j}
    + \frac12 \Delta \log r^{2\mu_j} \cdot \id_{\sigma^*F_j}
  = 0.
\end{align*}
Denote by $A_{\diamond,j}$ the Chern connection associated with $H_{\diamond,j}$ and by $B_j$ the Chern connection associated with $K_j$.
The isometry $r^{\mu_j} \co (\sigma^*F_j,H_{\diamond,j}) \to \sigma^*(F_j,K_j)$ transforms $A_{\diamond,j}$ into
\begin{equation*}
  A_{*,j} 
  := (r^{\mu_j})_* A_{\diamond,j}
   = \sigma^*B_j + i \mu_j \, \id_{\sigma^*F_j} \cdot \pi^*\theta.
\end{equation*}
In particular, $A_* := \bigoplus_{j=1}^k A_{*,j}$ is the pullback of a connection $B$ on $S^{2n-1}$;
moreover $A_*$ is unitary with respect to $H_* := \bigoplus_{j=1}^k \sigma^*K_j$.

\begin{prop}
  \label{Prop_HDiamondProperties}
  Assume the above situation.
  Set $H_\diamond := \bigoplus_{j=1}^k H_{\diamond,j}$ and fix $R > 0$.
  We have
  \begin{equation}
    \label{Eq_HDiamondCurvatureBounds}
    \Abs*{\abs{z}^{2+\ell}\nabla_{H_\diamond}^\ell \rF_{H_\diamond}}_{L^\infty(B_R(0))} < \infty
    \qforeach \ell \geq 0.
  \end{equation}
  Moreover, if $\sF$ is simple (that is $k = 1$), then
  \begin{equation}
    \label{Eq_HDiamondDirichletPoincare}
    \int_{\del B_r}\abs{s}^2
    \lesssim
    r^2 \int_{\del B_r} \abs{\nabla_{H_\diamond} s}^2
  \end{equation}
  for all $r \in (0,R]$ and $s \in C^\infty(\del B_r(0),i\su(\sigma^*F,H_\diamond))$.
\end{prop}

\begin{proof}
  Using the isometry $g := \bigoplus_{j=1}^k r^{\mu_j}$ both assertions can be translated to corresponding statements for $A_*$.
  The first assertion then follows since $A_*$ is the pullback of a connection $B$ on $S^{2n-1}$.
  If $k=1$, then $\nabla_B \co C^\infty(S^{2n-1},i\su(\rho^*F,K_1)) \to \Omega^1(S^{2n-1},i\su(\rho^*F,K_1))$ agrees with $\nabla_{\rho^*B_1}$ because $i \mu_1 \, \id_{\sigma^*F_1}$ is central.
  Therefore, any element of $\ker \nabla_B = \ker \nabla_{\rho^*B_1}$ must be invariant under the $S^1$--action and thus be the pullback of an element of $\ker \nabla_{B_1}$.
  The latter vanishes because $\sF_1$ is $\mu$--stable;
  hence, simple.
  This implies the second assertion.
\end{proof}

In the situation of \autoref{Thm_HYMConnectionTangentCones}, after a conformal change, which does not  affect $A^\circ$, we can assume that $\det H = \det H_\diamond$.
Setting
\begin{gather*}
  s := \log (H_\diamond^{-1}H) \in C^\infty(B_r(0)\setminus\set{0},i\su(\sigma^*F,H_\diamond)), \\
  \andq
  \Upsilon(s) := \frac{e^{\ad_s}-1}{\ad_s},
\end{gather*}
we have
\begin{gather*}
  e^{s/2}_*H = H_\diamond
  \qandq
  e^{s/2}_*A = A_\diamond + a \\
  \text{with}\quad
  a := \frac12\Upsilon(-s/2)\del_{A_\diamond}s - \frac12\Upsilon(s/2)\delbar_{A_\diamond} s;
\end{gather*}
see, e.g., \cite[Appendix A]{Jacob2016}.
Moreover, with $g = \bigoplus_{j=1}^k r^{\mu_j}$ we have
\begin{equation*}
  g_*e^{s/2}_* A
  = A_* + g a g^{-1}.
\end{equation*}
Since
\begin{equation*}
  \abs{\nabla_{A_*}^k gag^{-1}}_{H_*}
  =
    \abs{\nabla_{H_\diamond}^k a}_{H_\diamond}
    \qforeach k \geq 0,
\end{equation*}
\autoref{Thm_HYMConnectionTangentCones} will be a consequence of \autoref{Prop_HDiamondProperties} and the following result.

\begin{theorem}
  \label{Thm_HYMMetricTangentCones}
  Suppose $\omega = \frac{1}{2i} \delbar\del\abs{z}^2 + O(\abs{z}^2)$ is a Kähler form on $\bar B_{R}(0) \subset \C^n$,
  $\sE$ is a holomorphic vector bundle over $\bar B_R(0)\setminus\set{0}$,
  and $H_\diamond$ is a Hermitian metric on $\sE$ which is HYM with respect to $\omega_0$ and such that \eqref{Eq_HDiamondCurvatureBounds} holds.
  If $H$ is an admissible HYM metric on $\sE|_{B_R(0)}$ with $\sing(A_H) = \set{0}$ and $\det H = \det H_\diamond$, then $s := \log (H_\diamond^{-1}H) \in C^\infty(B_R(0)\setminus\set{0},i\su(\pi^*F,H_\diamond))$ satisfies
  \begin{equation*}
    \abs{s} \leq C_0 \qandq
    \abs{z}^k\abs{\nabla_{H_\diamond}^k s} \leq C_k(-\log \abs{z})^{-1/2}
    \qforeach k\geq 1.
  \end{equation*}
  Moreover, if \eqref{Eq_HDiamondDirichletPoincare} holds, then
  \begin{equation*}
    \abs{z}^k\abs{\nabla_{H_\diamond}^k s} \leq D_k \abs{z}^{\alpha}
    \qforeach k\geq 0.
  \end{equation*}
  The constants $C_k,D_k,\alpha > 0$ depend on $\omega$, $H_\diamond$, $s|_{B_{R}(0)\setminus B_{R/2}(0)}$, and $\Abs{\rF_H}_{L^2(B_R(0))}$.
\end{theorem}

The next three sections of this paper are devoted to proving \autoref{Thm_HYMMetricTangentCones}.
Without loss of generality, we will assume that the radius $R$ is one.

\paragraph{Conventions and notation.}

Set $B_r := B_r(0)$, $B := B_1$, and $\dot B := B\setminus\set{0}$.

We denote by $c > 0$ a generic constant, which depends only on $\sF$, $\omega$, $s|_{B_1\setminus B_{1/2}}$, $H_\diamond$, and $\Abs{\rF_H}_{L^2(B_R(0))}$.
Its value might change from one occurrence to the next.
Should $c$ depend on further data we indicate this by a subscript.
We write $x \lesssim y$ for $x \leq c y$. The expression $O(x)$ denotes a quantity $y$ with $\abs{y} \lesssim x$.

Since reflexive sheaves are locally free away from a closed subset of complex codimension three, without loss of generality, we will assume throughout that $n \geq 3$.

%%% Local Variables:
%%% mode: latex
%%% TeX-master: "HYMTangentCones"
%%% End:                   

\section{A priori \texorpdfstring{$C^0$}{C0} estimate}
\label{Sec_APrioriC0Estimate}

As a first step towards proving \autoref{Thm_HYMMetricTangentCones} we bound $\abs{s}$, using an argument which is essentially contained in \citet[Theorem 2(a) and (b)]{Bando1994}.

\begin{prop}
  \label{Prop_APrioriC0Bound}
  We have $\abs{s} \in L^\infty(B)$ and $\Abs{s}_{L^\infty(B)} \leq c$.
\end{prop}

\begin{proof}
  The proof relies on the differential inequality
  \begin{equation}
    \label{Eq_LogTrHSubsolution}
    \Delta \log \tr H_0^{-1}H_1 \lesssim \abs{\rK_{H_1} - \rK_{H_0}}
  \end{equation}
  for Hermitian metrics $H_0$ and $H_1$ with $\det H_0 = \det H_1$, and with
  \begin{equation*}
    \rK_H := i\Lambda \rF_H - \frac{\tr(i\Lambda \rF_H)}{\rk E}\cdot \id_E;
  \end{equation*}
  see \cite[p.~13]{Siu1987} for a proof.

  \setcounter{step}{0}
  \begin{step}
    We have $\log \tr e^s \in W^{1,2}(B)$ and $\Abs*{\log \tr e^s}_{W^{1,2}(B)} \leq c$.
  \end{step}

  Choose $1 \leq i < j \leq n$ and define the projection $\pi\co B \to \C^{n-2}$ by $\pi(z) := (z_1,\ldots, \hat z_i, \ldots \hat z_j,\ldots, z_n)$.
  For $\zeta \in \C^{n-2}$, denote by $\nabla_\zeta$ and $\Delta_\zeta$ the derivative and the Laplacian on the slice $\pi^{-1}(\zeta)$ respectively.
  Set $f_\zeta := \log \tr e^s|_{\pi^{-1}(\zeta)}$.
  Applying \eqref{Eq_LogTrHSubsolution} to $H|_{\pi^{-1}(\zeta)}$ and $H_\diamond|_{\pi^{-1}(\zeta)}$ we obtain 
  \begin{equation*}
    \Delta_\zeta f_\zeta
    \lesssim
    \abs{\rF_H} + \abs{\rF_{H_\diamond}}.
  \end{equation*}
  Fix $\chi \in C^\infty(\C^2;[0,1])$ such that $\chi(\eta) = 1$ for $\abs{\eta} \leq 1/2$ and $\chi(\eta) = 0$ for $\abs{\eta} \geq 1/\sqrt{2}$.
  For $0 < \abs{\zeta} \leq 1/\sqrt{2}$ and $\epsilon > 0$,  we have
  \begin{align*}
    \int_{\pi^{-1}(\zeta)} \abs{\nabla_\zeta (\chi f_\zeta)}^2
    &\lesssim
      \int_{\pi^{-1}(\zeta)} \chi^2 f_\zeta (\abs{\rF_H}+\abs{\rF_{H_\diamond}})
      + c \\
    &\leq
      \epsilon\int_{\pi^{-1}(\zeta)} \abs{\chi f_\zeta}^2
      + \epsilon^{-1} \int_{\pi^{-1}(\zeta)} \abs{\rF_H}^2+\abs{\rF_{H_\diamond}}^2
      + c.
  \end{align*}
  Using the Dirichlet--Poincaré inequality and rearranging, we obtain
  \begin{equation*}
    \int_{\pi^{-1}(\zeta)} \abs{\chi f_\zeta}^2 + \abs{\nabla_\zeta(\chi f_\zeta)}^2 
    \lesssim
    \int_{\pi^{-1}(\zeta)} \abs{\rF_H}^2+\abs{\rF_{H_\diamond}}^2 + c.
  \end{equation*}
  Integrating over $0 < \abs{\zeta} \leq 1/\sqrt{2}$ yields
  \begin{equation*}
    \int_B \abs{\log \tr e^s}^2 + \abs{\nabla' \log \tr e^s}^2 
    \lesssim
      \int_B \abs{\rF_H}^2+\abs{\rF_{H_\diamond}}^2 + c
  \end{equation*}
  with $\nabla'$ denoting the derivative along the fibres of $\pi$.
  Using \eqref{Eq_HDiamondCurvatureBounds} and $n \geq 3$, $\rF_{H_\diamond} \in L^2(B)$.
  Since the choice of $i,j$ defining $\pi$ was arbitrary, the asserted inequality follows.

  \begin{step}
    \label{Step_LogTrHWeakSubsolution}
    The differential inequality
    \begin{equation*}
      \Delta \log \tr e^s \lesssim \abs{\rK_{H_\diamond}}
    \end{equation*}
    holds on $B$ in the sense of distributions.
  \end{step}

  Fix a smooth function $\chi \co [0,\infty) \to [0,1]$ which vanishes on $[0,1]$ and is equal to one on $[2,\infty)$.
  Set $\chi_\epsilon := \chi(\abs{\cdot}/\epsilon)$.
  By \eqref{Eq_LogTrHSubsolution}, for $\phi \in C^\infty_0(B)$, we have
  \begin{align*}
    \int_B \Delta \phi \cdot \log \tr e^s
    &=
      \lim_{\epsilon\to 0} \int_B \chi_\epsilon \cdot \Delta \phi \cdot \log \tr e^s \\
    &\lesssim 
      \int_B \phi \cdot \abs{\rK_{H_\diamond}}
      + \lim_{\epsilon\to 0} \int_B \phi \cdot \(\Delta \chi_\epsilon \cdot \log \tr e^s - 2 \inner{\nabla\chi_\epsilon}{\nabla \log \tr e^s}\).
  \end{align*}
  Since $n \geq 3$, we have $\Abs{\chi_\epsilon}_{W^{2,2}(B)} \lesssim \epsilon^2$.
  Because $\log\tr e^s \in W^{1,2}(B)$ this shows that the limit vanishes.

  \begin{step}
    We have  $\log\tr e^s \in L^\infty(B)$ and $\Abs{\log\tr e^s}_{L^\infty(B)} \leq c$.
  \end{step}
  
  Since $\tr s = 0$, we have $\abs{s} \leq \rk(\sE) \cdot \log\tr e^s$;
  in particular, $\log \tr e^s$ is non-negative.
  By hypothesis $\rK_H = 0$.
  Since $H_\diamond$ is $\P$HYM with respect to $\omega_0$ and $\abs{\rF_{H_\diamond}} \lesssim \abs{z}^{-2}$ by hypothesis \eqref{Eq_HDiamondCurvatureBounds}, we have $\abs{\rK_{H_\diamond}} \leq c$.
  The asserted inequality thus follows from \autoref{Step_LogTrHWeakSubsolution} via Moser iteration;
  see \cite[Theorem 8.1]{Gilbarg2001}.
\end{proof}

%%% Local Variables:
%%% mode: latex
%%% TeX-master: "HYMTangentCones"
%%% End:

\section{A priori Morrey estimates}
\label{Sec_APrioriMorreyEstimates}

The following decay estimates are the crucial ingredient of the proof of \autoref{Thm_HYMMetricTangentCones}.

\begin{prop}
  \label{Prop_LogarithmicMorreyBound}
  For $r \in [0,1]$ we have
  \begin{equation*}
     \int_{B_r} \abs{\nabla_{H_\diamond} s}^2 \lesssim r^{2n-2}(-\log r)^{-1}.
  \end{equation*}
\end{prop}

\begin{prop}
  \label{Prop_MorreyBound}
  If \eqref{Eq_HDiamondDirichletPoincare} holds, then there is a constant $\alpha > 0$, depending on $\Abs{s}_{L^\infty(B)}$ in a monotone decreasing way, such that for $r \in [0,1]$ we have
  \begin{equation*}
    \int_{B_r} \abs{s}^2 \lesssim r^{2n+2\alpha} \qandq
    \int_{B_r} \abs{\nabla_{H_\diamond} s}^2 \lesssim r^{2n-2+2\alpha}.
  \end{equation*}
\end{prop}

Both of these results rely on the following inequality.

\begin{prop}
  \label{Prop_NablaS}  
  We have
  \begin{equation*}
    \abs{\nabla_{H_\diamond} s}^2
    \lesssim
    1 - \Delta \abs{s}^2.
  \end{equation*}
\end{prop}

\begin{proof}
  Since $H = H_\diamond e^s$ is $\P$HYM, we have
  \begin{equation*}
    \Delta \abs{s}^2 + 2\abs{\upsilon(-s)\nabla_{H_\diamond}s}^2
    \leq
      -4\inner{\rK_{H_\diamond}}{s}
  \end{equation*}
  with
  \begin{equation*}
    \upsilon(-s)
    =
      \sqrt{\frac{1-e^{-\ad_s}}{\ad_s}} \in \End(\gl(E));
  \end{equation*}
  see, e.g., \cite[Proposition A.6]{Jacob2016}.
  The assertion follows using
  \begin{equation*}
    \sqrt{\frac{1-e^{-x}}{x}}
    \gtrsim
      \frac{1}{\sqrt{1+\abs{x}}},
  \end{equation*}
  $\Abs{\rK_{H_\diamond}}_{L^\infty} \leq c$, which is a consequence of \eqref{Eq_HDiamondCurvatureBounds}, and the bound on $\abs{s}$ established in \autoref{Prop_APrioriC0Bound}.
\end{proof}

\begin{proof}[Proof of \autoref{Prop_MorreyBound}]
  Define $g \co [0,1/2] \to [0,\infty]$ by
  \begin{equation*}
    g(r)
    :=
    \int_{B_r} \abs{z}^{2-2n}\abs{\nabla_{H_\diamond}s}^2.
  \end{equation*} 
  We will show that
  \begin{equation*}
    g(r) \leq cr^{2\alpha},
  \end{equation*}
  which implies the second asserted inequality and using \eqref{Eq_HDiamondDirichletPoincare} also the first.

  \setcounter{step}{0}
  \begin{step}
    \label{Step_GBounded}
    We have $g \leq c$.
  \end{step}

  Fix a smooth function $\chi \co [0,\infty) \to [0,1]$ which is equal to one on $[0,1]$ and vanishes outside $[0,2]$.
  Set $\chi_{r}(\cdot) := \chi(\abs{\cdot}/r)$.
  For $r > \epsilon > 0$, using \autoref{Prop_NablaS} and \autoref{Prop_APrioriC0Bound}, and with $G$ denoting Green's function on $B$ centered at $0$, we have
  \begin{align*}
    \int_{B_r\setminus B_\epsilon} \abs{z}^{2-2n}\abs{\nabla_{H_\diamond}s}^2
    &\lesssim
      \int_{B_{2r}\setminus B_{\epsilon/2}} \chi_r(1-\chi_{\epsilon/2}) G(1-\Delta\abs{s}^2) \\
    &\lesssim
      \int_{B_{2r}\setminus B_r} \abs{z}^{-2n}\abs{s}^2
      + r^2
      + \epsilon^{-2n}\int_{B_\epsilon\setminus B_{\epsilon/2}} \abs{s}^2 \\
    &\leq
      c.
  \end{align*}

  \begin{step}
    \label{Step_GDoublingBound}
    There are constants $\gamma \in [0,1)$ and $A > 0$ such that
    \begin{equation*}
      g(r) \leq \gamma g(2r) + A r^2.
    \end{equation*}
  \end{step}

  Continuing the inequality from \autoref{Step_GBounded} using \eqref{Eq_HDiamondDirichletPoincare}, we have
  \begin{align*}
    \int_{B_r\setminus B_\epsilon} \abs{z}^{2-2n}\abs{\nabla_{H_\diamond}s}^2
    &\lesssim
      \int_{B_{2r}\setminus B_r} \abs{z}^{2-2n}\abs{\nabla_{H_\diamond} s}^2
      + r^2
      + \epsilon^{2-2n}\int_{B_\epsilon\setminus B_{\epsilon/2}} \abs{\nabla_{H_\diamond} s}^2 \\
    &\lesssim
      g(2r) - g(r)
      + r^2
      + g(\epsilon).
  \end{align*}
  By Lebesgue's monotone convergence theorem, the last term vanishes as $\epsilon$ tends to zero;
  hence, the asserted inequality follows with $\gamma = \frac{c}{c+1}$ and $A = c$.

  \begin{step}
    \label{Step_GGrowthBound}
    We have $g \leq cr^{2\alpha}$ for some $\alpha \in (0,1)$.
  \end{step}

  This follows from \autoref{Step_GBounded} and \autoref{Step_GDoublingBound} and as in \cite[Step 3 in the proof of Proposition C.2]{Jacob2016}.
\end{proof}

\begin{proof}[Proof of \autoref{Prop_LogarithmicMorreyBound}]
  We use the same notation as in the proof of \autoref{Prop_MorreyBound}.
  It still holds that $g \leq c$.
  However, the proof of the doubling estimate in \autoref{Step_GDoublingBound} uses that $\sF$ is simple and will not carry over.
  Instead, using integration by parts and Hölder's inequality we have
  \begin{align*}
    \int_{B_r\setminus B_\epsilon} \abs{z}^{2-2n}\abs{\nabla_{H_\diamond}s}^2
    &\lesssim
      \int_{B_{2r}\setminus B_{\epsilon/2}} \chi_r(1-\chi_{\epsilon/2}) G(1-\Delta\abs{s}^2) \\
    &\lesssim
      \int_{B_{2r}\setminus B_r} \abs{z}^{1-2n}\del_r \abs{s}^2
      + r^2
      + \epsilon^{1-2n}\int_{B_\epsilon\setminus B_{\epsilon/2}} \del_r\abs{s}^2 \\
    &\lesssim
      \(\int_{B_{2r}\setminus B_r} \abs{z}^{2-2n}\abs{\nabla_{H_\diamond}s}^2\)^{1/2}
      + r^2 \\
    &\quad\quad
      + \(\int_{B_\epsilon\setminus B_{\epsilon/2}} \abs{z}^{2-2n}\abs{\nabla_{H_\diamond}s}^2\)^{1/2}.
  \end{align*}
  By Lebesgue's monotone convergence theorem, the last term vanishes as $\epsilon$ tends to zero;
  hence,
  \begin{equation*}
    g(r) \lesssim (g(2r)-g(r))^{1/2} + r^2.
  \end{equation*}
  The asserted inequality now follows from \autoref{Prop_LogarithmicGrowthEstimate}. 
\end{proof}

\begin{prop}
  \label{Prop_LogarithmicGrowthEstimate}
  If $g \co [0,1] \to [0,\infty)$ is monotone increasing and satisfies
  \begin{equation*}
    g(r) \leq A (g(2r)-g(r))^{1/2} + Br^2,
  \end{equation*}
  then there are constants $c > 0$ and $r_0 \in (0,1]$ depending on $A$, $B$ and $g(1)$, such that
  \begin{equation*}
    g(r) \lesssim c(-\log r)^{-1}
  \end{equation*}
  for $r \in (0,r_0]$.
\end{prop}

\begin{proof}
  For $r \in (0,r_0]$ the function $h(r) := g(r)+B/Ar^2$ satisfies
  \begin{equation*}
    h(r)^2 \leq 2A(h(2r)-h(r));
  \end{equation*}
  hence,
  \begin{equation*}
    h(r) \leq \frac{1}{1+\epsilon h(r)} h(2r)
  \end{equation*}
  with $\epsilon = 1/2A$.
  We can assume that $\epsilon h(1) \leq 1/2$.
  Using $(1+x)^{-1} \leq 1 - x$ for $x \geq 0$, and $(1-x)^k \leq 1 - \frac{k}{2}x$ for $x \in [0,1/2]$, we derive
  \begin{equation*}
    0
    \leq
    h(2^{-k})
    \leq
    \(1-\frac{k\epsilon}{2} h(2^{-k})\)h(1);
  \end{equation*}
  hence,
  \begin{equation*}
    h(2^{-k})
    \leq
      \frac{2}{\epsilon k}.
    \qedhere
  \end{equation*}
\end{proof}

%%% Local Variables:
%%% mode: latex
%%% TeX-master: "HYMTangentCones"
%%% End:

\section{Proof of \autoref{Thm_HYMMetricTangentCones}}

For $r > 0$, define $m_r \co \C^n \to \C^n$ by $m_r(z) := rz$.
Set
\begin{equation*}
  s_r := m_r^*(s|_{B_{4r}\setminus B_{r/2}}) \in C^\infty(B_4\setminus B_{1/2},i\su(E,H_*))
  \qandq
  H_{\diamond,r} := m_r^*H_\diamond.
\end{equation*}
The metric $H_{\diamond,r} e^{s_r}$ is $\P$HYM with respect to $\omega_r := r^{-2}m_r^*\omega $ and $\Abs{\rF_{H_\diamond,r}}_{C^k(B_4\setminus B_{1/2})} \leq c_k$.

\autoref{Prop_APrioriC0Bound}, \eqref{Eq_HDiamondCurvatureBounds} and interior estimates for $\P$HYM metrics \cite[Theorem C.1]{Jacob2016} imply that
\begin{equation*}
  \Abs{s_r}_{C^k(B_3\setminus B_{3/4})} \leq c_k.
\end{equation*}
By \autoref{Prop_LogarithmicMorreyBound}, we have
\begin{equation*}
  \Abs{\nabla_{H_{\diamond,r}} s_r}_{L^2(B_3\setminus B_{3/4})}
  \leq
  c_k (-\log r)^{-1/2}.
\end{equation*}
Schematically, $\rK_{H_{\diamond,r} e^{s_r}} = 0$ can be written as
\begin{equation*}
  \nabla_{H_{\diamond,r}}^*\nabla_{H_{\diamond,r}} s_r
  + B(\nabla_{H_{\diamond,r}} s\otimes \nabla_{H_{\diamond,r}} s_r)
  = C(\rK_{H_{\diamond,r}}),
\end{equation*}
where $B$ and $C$ are linear with coefficients depending on $s$, but not on its derivatives; see, e.g., \cite[Proposition A.1]{Jacob2016}.
Since $\Abs{\rK_{H_{\diamond,r}}}_{C^k(B_3\setminus B_{3/4})} \leq c_kr^2$, as in \cite[Step 3 in the proof of Proposition 5.1]{Jacob2016}, standard interior estimates imply that
\begin{equation*}
  \Abs{\nabla_{H_{\diamond,r}}^k s_r}_{L^\infty(B_2\setminus B_1)}
  \leq
  c_k (-\log r)^{-1/2}
\end{equation*}
and, hence, the asserted inequalities, for each $k \geq 1$.
(The asserted inequality for $k = 0$ has already be proven in \autoref{Prop_APrioriC0Bound}.)

If \eqref{Eq_HDiamondDirichletPoincare} holds, then by \autoref{Prop_MorreyBound} we have
\begin{equation*}
  \Abs{\nabla_{H_{\diamond,r}} s_r}_{L^2(B_4\setminus B_{1/2})}
  \lesssim
    r^{\alpha}
  \qandq
  \Abs{s_r}_{L^2(B_4\setminus B_{1/2})}
  \lesssim
    r^{\alpha};
\end{equation*}
hence, using standard interior estimates
\begin{equation*}
  \Abs{\nabla_{H_{\diamond,r}}^k s_r}_{L^2(B_2\setminus B_{1})}
  \lesssim
    r^{\alpha}
    \qforeach k \geq 0.
\end{equation*}
This concludes the proof of \autoref{Thm_HYMMetricTangentCones}.
\qed

%%% Local Variables:
%%% mode: latex
%%% TeX-master: "HYMTangentCones"
%%% End:

\section{Proof of \autoref{Prop_RadiallyInvariantHYMConnections}}
\label{Sec_RadiallyInvariantHYMConnections}

We will make use of the following general fact about connections over manifolds with free $S^1$--actions.

\begin{prop}
  \label{Prop_ConnectionsOnS1Spaces}
  Let $M$ be a manifold with a free $S^1$--action.
  Denote the associated Killing field by $\xi \in \Vect(M)$ and let $q\co M \to M/S^1$ be the canonical projection.
  Suppose $\theta \in \Omega^1(M)$ is such that $\theta(\xi) = 1$ and $\sL_\xi\theta = 0$.
  Let $A$ be a unitary connection on a Hermitian vector bundle $(E,H)$ over $M$.
  If $i(\xi)\rF_A = 0$, then there is a $k \in \N$ and, for each $j \in \set{1,\ldots,k}$, a Hermitian vector bundles $(F_j,K_j)$ over $M/S^1$ such that
  \begin{equation*}
    E = \bigoplus_{j=1}^k E_j \qandq
    H = \bigoplus_{j=1}^k H_j
  \end{equation*}
  with $E_j := q^*F_j$ and $H_j := q^*K_j$;
  moreover, the bundles $E_j$ are parallel and, for each $j \in \set{1,\ldots,k}$, there are a unitary connection $B_j$ on $F_j$ and $\mu_j \in \R$ such that
  \begin{equation*}
    A = \bigoplus_{j=1}^k q^*B_j + i\mu_j\,\id_{E_j} \cdot \theta.
  \end{equation*}
\end{prop}

\begin{proof}
  Denote by $\tilde\xi \in \Vect(\U(E))$ the $A$--horizontal lift of $\xi$.
  This vector field integrates to an $\R$--action of $\U(E)$.
  Thinking of $A$ as an $\fu(r)$--valued $1$--form on $\U(E)$ and $\rF_A$ as an $\fu(r)$--valued $2$--form on $\U(E)$, we have
  \begin{equation*}
    \sL_{\tilde \xi} A = i(\tilde\xi) \rF_A = 0;
  \end{equation*}
  hence, $A$ is invariant with respect to the $\R$--action on $\U(E)$.

  The obstruction to the $\R$--action on $U(E)$ inducing an $S^1$--action is the action of $1 \in \R$ and corresponds to a gauge transformation $\bg_A \in \sG(\U(E))$ fixing $A$.
  If this obstruction vanishes, i.e., $\bg_A = \id_{\U(E)}$, then $E \iso q^*F$ with $F = E/S^1$ and there is a connection $A_0$ on $F$ such that $A = q^*A_0$.

  If the obstruction does not vanish, we can decompose $E$ into pairwise orthogonal parallel subbundles $E_j$ such that $\bg_A$ acts on $E_j$ as multiplication with $e^{i \mu_j}$ for some $\mu_j \in \R$.
  Set $\tilde A := A - \bigoplus_{j=1}^k i\mu_j \, \id_{E_j} \cdot \theta$.
  This connection also satisfies $i(\xi) \rF_{\tilde A} = 0 \in \Omega^1(M,\fg_E)$ and the subbundles $E_j$ are also parallel with respect to $E_j$.
  Since $\bg_{\tilde A} = \id_E$, the assertion follows.
\end{proof}

In the situation of \autoref{Prop_RadiallyInvariantHYMConnections}, with $\xi \in S^{2n-1}$ denoting the Killing field for the $S^1$--action we have $i(\xi)\rF_{A_0} = 0$; c.f., \citet[discussion after Conjecture 2]{Tian2000}.
Therefore, we can write
\begin{equation*}
  A_* = \bigoplus_{j=1}^k \sigma^*B_j + i\mu_j\,\id_{E_j} \cdot \pi^*\theta.
\end{equation*}
Since $\rd\theta = 2\pi \rho^*\omega_{FS}$, we have
\begin{equation*}
  \rF_{A_*} = \bigoplus_{j=1}^k \sigma^*\rF_{B_j} + 2\pi i\mu_j\,\id_{E_j} \cdot \sigma^*\omega_{FS}.
\end{equation*}
Using \eqref{Eq_Omega0Expansion}, $A_*$ being HYM with respect to $\omega_0$ can be seen to be equivalent to
\begin{equation*}
  \rF_{B_j}^{0,2} = 0 \qandq
  i\Lambda \rF_{B_j} = (2n-2)\pi \mu_j \cdot \id_{E_j}.
\end{equation*}
The isomorphism $\sE = (E,\delbar_{A_*}) \iso \bigoplus_{j=1}^k \rho^*\sF_j$ with $\sF_j = (F_j,\delbar_{B_j})$ is given by $g^{-1}$ with $g := \bigoplus_{j=1}^k r^{\mu_j}$.
\qed

%%% Local Variables:
%%% mode: latex
%%% TeX-master: "HYMTangentCones"
%%% End:                   

\paragraph{Acknowledgements.}

HSE and TW were partially supported by São Paulo State Research Council
(FAPESP) grant 2015/50368-0 and the MIT--Brazil Lemann Seed Fund for Collaborative Projects.
HSE is  also funded by FAPESP grant 2014/24727-0 and Brazilian National Research Council (CNPq) grant PQ2 - 312390/2014-9. 

\printreferences

\end{document}